\documentclass[reqno]{amsproc}%
\usepackage{amssymb}
\usepackage{amsfonts}%
\usepackage{amsmath}%
\usepackage{graphicx}
\providecommand{\U}[1]{\protect\rule{.1in}{.1in}}
\theoremstyle{plain}

\newtheorem{definition}{Definition}

\newtheorem{lemma}{Lemma}

\newtheorem{theorem}{Theorem}
\numberwithin{equation}{section}
\begin{document}
\title{ Dunkl-Gamma Type Operators including Appell Polynomials}
\author{Fatma Ta\c{s}delen}
\address{Ankara University, Faculty of Science, Department of Mathematics, 06100,
Tando\u{g}an, Ankara, Turkey}
\email{tasdelen@science.ankara.edu.tr}
\author{Dilek S\"{o}ylemez}
\address{Ankara University, Department of Computer Programming, Elmadag Vocational
School, Ankara, Turkey}
\email{dsoylemez@ankara.edu.tr}
\author{Rabia Akta\c{s}}
\address{Ankara University, Faculty of Science, Department of Mathematics, 06100,
Tando\u{g}an, Ankara, Turkey}
\email{raktas@science.ankara.edu.tr}

\begin{abstract}
The aim of the present paper is to introduce Dunkl-Gamma type operators in
terms of Appell polynomials and to investigate approximating properties of
these operators.

\end{abstract}
\maketitle

\section{Introduction}

Recently, linear positive operators constructed via generating functions and
their further extentions are intensively studied by many research authors, for
example, we refer the readers to \cite{ismail, J-L, kajla-agrawal, MMS,
mursaleen-ensari2015, olgun-ince-tasdelen,tasdelen-aktas-altin, VIS, VTT}. In
\cite{J-L}, Jakimovski et al. introduced linear positive operators in terms
of Appell polynomials as follows%
\begin{equation}
\left(  T_{n}f\right)  \left(  x\right)  =\frac{e^{-nx}}{g\left(  1\right)
}\sum\limits_{k=0}^{\infty}p_{k}\left(  nx\right)  f\left(  \frac{k}{n}\right)
\label{*}%
\end{equation}
under the assumption
\begin{equation}
\dfrac{a_{k}}{g\left(  1\right)  }\geq0,~k=0,1,2,...~\label{**}%
\end{equation}
where $g\left(  1\right)  \neq0.$ Here, Appell polynomials $p_{k}\left(
x\right)  $ are generated by%
\[
g\left(  t\right)  e^{xt}=\sum\limits_{k=0}^{\infty}p_{k}\left(  x\right)
t^{k},
\]
where $g\left(  t\right)  $ is an analytic function in the disc $\left\vert
t\right\vert <R~~(R>1)$%
\[
g\left(  t\right)  =\sum\limits_{r=0}^{\infty}a_{r}t^{r}~~,~~a_{0}\neq0
\]
(see \cite{C}). In \cite{CI}, Ciupa defined the following Durrmeyer type
integral modification of the operators (\ref{*})
\begin{equation}
\left(  P_{n}f\right)  \left(  x\right)  =\frac{e^{-nx}}{g\left(  1\right)
}\sum\limits_{k=0}^{\infty}p_{k}\left(  nx\right)  \frac{n^{\lambda+k+1}%
}{\Gamma\left(  \lambda+k+1\right)  }\int\limits_{0}^{\infty}e^{-nt}%
t^{\lambda+k}f\left(  t\right)  dt\label{2}%
\end{equation}
under the assumption given by (\ref{**})$~$where $\lambda\geq0.$ Sucu
\cite{sucu} introduced Dunkl analogue of the Szasz operators by
\[
S_{n}^{\ast}\left(  f;x\right)  =\frac{1}{e_{\mu}\left(  nx\right)  }%
\sum_{k=0}^{\infty}\frac{\left(  nx\right)  ^{k}}{\gamma_{\mu}\left(
k\right)  }f\left(  \frac{k+2\mu\theta_{k}}{n}\right)  ,\text{ }n\in%
\mathbb{N}%
\]
for any $x\in\left[  0,\infty\right)  ,$ $n\in\mathbb{N}$, $\mu\geq0$ and
$f\in C\left[  0,\infty\right)  $ by using Dunkl generalization of the
exponential function $e_{\mu}\left(  x\right)  $ defined by \cite{rosenblum}%
\[
e_{\mu}\left(  x\right)  =\sum_{k=0}^{\infty}\frac{x^{k}}{\gamma_{\mu}\left(
k\right)  },
\]
where the coefficients $\gamma_{\mu}$ are in the form
\begin{equation}
\gamma_{\mu}\left(  2k\right)  =\frac{2^{2k}k!}{\Gamma\left(  k+\mu+\frac
{1}{2}\right)  }\text{ and }\gamma_{\mu}\left(  2k+1\right)  =\frac
{2^{2k+1}k!\Gamma\left(  k+\mu+\frac{3}{2}\right)  }{\Gamma\left(  \mu
+\frac{1}{2}\right)  }\label{1b}%
\end{equation}
for $k\in%
\mathbb{N}
_{0},$ $\mu>-\frac{1}{2}.$ Moreover, the next recursion formula is satisfied%
\begin{equation}
\gamma_{\mu}\left(  k+1\right)  =\left(  k+1+2\mu\theta_{k+1}\right)
\gamma_{\mu}\left(  k\right)  ,\text{ }k\in%
\mathbb{N}
_{0},\label{1c}%
\end{equation}
where $\theta_{k}$ is%
\[
\theta_{k}=\left\{
\begin{array}
[c]{cc}%
0, & if\text{ }k=2p\\
1, & if\text{ }k=2p+1
\end{array}
\right.  .
\]
Now, let us recall the Dunkl derivative operator \cite{dunkl1991,dunkl1994}.
\textbf{\newline}Let $\mu$ be a real number satisfying $\mu>-\frac{1}{2}.$
The Dunkl operator $T_{\mu}$ is defined by%
\[
T_{\mu}~\phi\left(  x\right)  =\phi^{\prime}\left(  x\right)  +\mu\frac
{\phi\left(  x\right)  -\phi\left(  -x\right)  }{x},
\]
where $\phi\left(  x\right)  $ is an entire function. For $\mu=0,$ the
operator $T_{\mu}$ gives the derivative operator. It is clear that%
\begin{equation}
T_{\mu}e_{\mu}\left(  xt\right)  =te_{\mu}\left(  xt\right)  ,\label{a}%
\end{equation}%
\begin{equation}
T_{\mu}x^{n}=\frac{\gamma_{\mu}\left(  n\right)  }{\gamma_{\mu}\left(
n-1\right)  }x^{n-1}.\label{b}%
\end{equation}
Moreover, the Dunkl generalization of the product of two function is given by%
\begin{equation}
T_{\mu}\left(  fg\right)  \left(  x\right)  =f\left(  x\right)  T_{\mu
}g\left(  x\right)  +g\left(  -x\right)  T_{\mu}f\left(  x\right)  +f^{\prime
}\left(  x\right)  \left[  g\left(  x\right)  -g\left(  -x\right)  \right]
,\label{c}%
\end{equation}
which gives the next result if the function $g$ is an even function
\[
T_{\mu}\left(  fg\right)  \left(  x\right)  =f\left(  x\right)  T_{\mu
}g\left(  x\right)  +g\left(  x\right)  T_{\mu}f\left(  x\right)  .
\]

By the motivation this work, many authors studied Dunkl analogue of the
several approximation operators for example, we refer the readers to
\cite{aktas-cekim-tasdelen, ben cheikh2014, deshwal-agrawal-araci,
icoz-cekim2015, icoz, Rao-wafi-ana acu d}.

Wafi and Rao \cite{wafi-rao} constructed Dunkl analogue of Szasz Durrmeyer
operators as%
\begin{equation}
D_{n}\left(  f;x\right)  =\frac{1}{e_{\mu}\left(  nx\right)  }\sum
_{k=0}^{\infty}\frac{\left(  nx\right)  ^{k}}{\gamma_{\mu}\left(  k\right)
}\frac{n^{k+2\mu\theta_{k}+\lambda+1}}{\Gamma\left(  k+2\mu\theta_{k}%
+\lambda+1\right)  }%
{\displaystyle\int\limits_{0}^{\infty}}
e^{-nt}t^{^{k+2\mu\theta_{k}+\lambda}}f\left(  t\right)  dt\label{dunkl-s-d}%
\end{equation}
for any $\lambda\geq0,$ $x\in\left[  0,\infty\right)  ,$ $n\in\mathbb{N}$,
$\mu\geq0$ and $f\in C\left[  0,\infty\right)  .$ The authors also examined
pointwise approximation results in several functional spaces. They also
studied weighted approximation results and gave rate of convergence for
functions with derivative of bounded variation.

In \cite{ben cheikh 2007}, Ben Cheikh studied some properties of
Dunkl-Appell $d-$ortogonal polynomials. In that work, Dunkl-Appell polynomials
$p_{k}\left(  x\right)  $ defined by%
\[
p_{k}\left(  x\right)  =\sum\limits_{n=0}^{k}\binom{k}{n}_{\mu}a_{k-n}%
x^{n},\text{ }\left(  a_{k}\right)  _{k\geq0}%
\]
are generated by%

\begin{equation}
A\left(  t\right)  e_{\mu}\left(  xt\right)  =\sum_{k=0}^{\infty}\frac
{p_{k}\left(  x\right)  }{\gamma_{\mu}\left(  k\right)  }t^{k}%
,\label{dunkl-appell}%
\end{equation}
where $A\left(  t\right)  $ is an analytic function in the disc $\left\vert
t\right\vert <R~~(R>1)$%
\begin{equation}
A\left(  t\right)  =\sum\limits_{r=0}^{\infty}\frac{a_{r}}{\gamma_{\mu}\left(
r\right)  }t^{r},~~a_{0}\neq0\label{***}%
\end{equation}
and Dunkl-binomial coefficient is defined by%
\[
\binom{k}{n}_{\mu}=\frac{\gamma_{\mu}\left(  k\right)  }{\gamma_{\mu}\left(
n\right)  \gamma_{\mu}\left(  k-n\right)  }.
\]
Note that $\gamma_{0}\left(  k\right)  =k!$ and $\binom{k}{n}_{0}=\binom{k}%
{n}.$ \newline

Inspired by the above works, for any $x\in\left[  0,\infty\right)  ,~f\in
C\left[  0,\infty\right)  ,$ we introduce Dunkl analogue of the Appell Szasz
Durrmeyer operators as%

\begin{equation}
D_{n}^{\ast}\left(  f;x\right)  =\frac{1}{e_{\mu}\left(  nx\right)  A\left(
1\right)  }\sum_{k=0}^{\infty}\frac{p_{k}\left(  nx\right)  }{\gamma_{\mu
}\left(  k\right)  }\frac{n^{k+2\mu\theta_{k}+\lambda+1}}{\Gamma\left(
k+2\mu\theta_{k}+\lambda+1\right)  }%
{\displaystyle\int\limits_{0}^{\infty}}
e^{-nt}t^{^{k+2\mu\theta_{k}+\lambda}}f\left(  t\right)
dt,\label{appell-dunkl-s-d}%
\end{equation}
where $\mu,\lambda\geq0,~A\left(  1\right)  \neq0,~~\frac{a_{k-n}}{A\left(
1\right)  }\geq0,~\left(  0\leq n\leq k\right)  ,~~k=0,1,2,...,$ and
$\gamma_{\mu}$ is defined by $\left(  \ref{1b}\right)  $ and $A\left(
t\right)  $ is given as in $\left(  \ref{***}\right)  .$

Note that in the case of $\mu=0$ the operator $\left(  \ref{appell-dunkl-s-d}%
\right)  $ gives the operator $\left(  \ref{2}\right)  $, and for $A\left(
t\right)  =1$ the operator $\left(  \ref{appell-dunkl-s-d}\right)  $ reduces
to the operator $\left(  \ref{dunkl-s-d}\right)  .$

We organize the paper as follows. In section 2, we give some lemmas and obtain
the convergence of the operators $\left(  \ref{appell-dunkl-s-d}\right)  $
with the help of universal Korovkin-type theorem. In section 3, we compute the
rates of convergence of the operators $D_{n}^{\ast}\left(  f\right)  $ to $f$
by means of the usual and second modulus of continuity and Lipschitz class functions.

\section{Approximation properties of the operators $D_{n}^{\ast}$}

In what follows, we first give some lemmas and then prove the main theorem
with the help of the well-known Korovkin Theorem.

\begin{lemma}
\label{lemma 1}From the generating function (\ref{dunkl-appell}), the
following equalities are satisfied%
\begin{equation}%
{\displaystyle\sum\limits_{k=0}^{\infty}}
\frac{p_{k}\left(  nx\right)  }{\gamma_{\mu}\left(  k\right)  }=A\left(
1\right)  e_{\mu}\left(  nx\right)  ,\label{1}%
\end{equation}%
\begin{align}
&
{\displaystyle\sum\limits_{k=0}^{\infty}}
\frac{p_{k+1}\left(  nx\right)  }{\gamma_{\mu}\left(  k\right)  }=\left(
nx\right)  A\left(  1\right)  e_{\mu}\left(  nx\right) \nonumber\\
& +\mu e_{\mu}\left(  -nx\right)  \left[  A\left(  1\right)  -A\left(
-1\right)  \right]  +A^{\prime}\left(  1\right)  e_{\mu}\left(  nx\right)
\label{2.2}%
\end{align}
and%
\begin{align}
&
{\displaystyle\sum\limits_{k=0}^{\infty}}
\frac{p_{k+2}\left(  nx\right)  }{\gamma_{\mu}\left(  k\right)  }=n^{2}%
x^{2}A\left(  1\right)  e_{\mu}\left(  nx\right)  +2nxe_{\mu}\left(
nx\right)  A^{\prime}\left(  1\right) \nonumber\\
& +2\mu e_{\mu}\left(  -nx\right)  \left[  A^{\prime}\left(  1\right)
-\frac{A\left(  1\right)  -A\left(  -1\right)  }{2}\right] \nonumber\\
& +A^{\prime\prime}\left(  1\right)  e_{\mu}\left(  nx\right)  .\label{2.3*}%
\end{align}

\end{lemma}

\begin{proof}
Taking $t\rightarrow1,$ $x\rightarrow nx$ in (\ref{dunkl-appell}), we get the
first one. When we apply the Dunkl operator $T_{\mu}$ to both of sides of the
equality $\left(  \ref{dunkl-appell}\right)  $, by using the relations
(\ref{a}), (\ref{b}) and (\ref{c}) we obtain the second and third relations.
\end{proof}

\begin{lemma}
\label{lemma 2.2}For the operators $D_{n}^{\ast},\ $one can have%
\[
D_{n}^{\ast}\left(  1;x\right)  =1
\]%
\[
D_{n}^{\ast}\left(  t;x\right)  =x+\frac{\mu}{n}\left\{  \frac{e_{\mu}\left(
-nx\right)  }{e_{\mu}\left(  nx\right)  }\left[  \frac{A\left(  1\right)
-A\left(  -1\right)  }{A\left(  1\right)  }\right]  \right\}  +\frac{1}%
{n}\left[  \frac{A^{^{\prime}}\left(  1\right)  }{A\left(  1\right)  }%
+\lambda+1\right]
\]%
\begin{align}
D_{n}^{\ast}\left(  t^{2};x\right)   & =x^{2}+\frac{x}{n}\left\{  2\mu
\frac{e_{\mu}\left(  -nx\right)  }{e_{\mu}\left(  nx\right)  }\frac{A\left(
-1\right)  }{A\left(  1\right)  }+\frac{2A^{^{\prime}}\left(  1\right)
}{A\left(  1\right)  }+2\lambda+4\right\} \label{2.0-c}\\
& +\frac{2\mu}{n^{2}}\frac{e_{\mu}\left(  -nx\right)  }{e_{\mu}\left(
nx\right)  }\frac{A^{^{\prime}}\left(  1\right)  +A^{^{\prime}}\left(
-1\right)  }{A\left(  1\right)  }+\frac{2\mu^{2}}{n^{2}}\left[  \frac{A\left(
1\right)  -A\left(  -1\right)  }{A\left(  1\right)  }\right] \nonumber\\
& +\frac{\mu}{n^{2}}\left(  2\lambda+3\right)  \frac{e_{\mu}\left(
-nx\right)  }{e_{\mu}\left(  nx\right)  }\left[  \frac{A\left(  1\right)
-A\left(  -1\right)  }{A\left(  1\right)  }\right] \nonumber\\
& +\frac{1}{n^{2}}\left[  \frac{A^{\prime\prime}\left(  1\right)  }{A\left(
1\right)  }+\left(  2\lambda+4\right)  \frac{A^{\prime}\left(  1\right)
}{A\left(  1\right)  }+\left(  \lambda+1\right)  \left(  \lambda+2\right)
\right]  .\nonumber
\end{align}

\end{lemma}

\begin{proof}
For $f\left(  t\right)  =1$ in the operator (\ref{appell-dunkl-s-d}), we have
\begin{align*}
D_{n}^{\ast}\left(  1;x\right)   & =\frac{1}{e_{\mu}\left(  nx\right)
A\left(  1\right)  }\sum_{k=0}^{\infty}\frac{p_{k}\left(  nx\right)  }%
{\gamma_{\mu}\left(  k\right)  }\frac{n^{k+2\mu\theta_{k}+\lambda+1}}%
{\Gamma\left(  k+2\mu\theta_{k}+\lambda+1\right)  }%
{\displaystyle\int\limits_{0}^{\infty}}
e^{-nt}t^{^{k+2\mu\theta_{k}+\lambda}}f\left(  t\right)  dt\\
& =\frac{1}{e_{\mu}\left(  nx\right)  A\left(  1\right)  }\sum_{k=0}^{\infty
}\frac{p_{k}\left(  nx\right)  }{\gamma_{\mu}\left(  k\right)  },
\end{align*}
from $\left(  \ref{1}\right)  ,$ it follows $D_{n}^{\ast}\left(  1;x\right)
=1. $ For $f\left(  t\right)  =t,$ the operator (\ref{appell-dunkl-s-d})
reduces to
\begin{align}
D_{n}^{\ast}\left(  t;x\right)   & =\frac{1}{e_{\mu}\left(  nx\right)
A\left(  1\right)  }\sum_{k=0}^{\infty}\frac{p_{k}\left(  nx\right)  }%
{\gamma_{\mu}\left(  k\right)  }\frac{n^{k+2\mu\theta_{k}+\lambda+1}}%
{\Gamma\left(  k+2\mu\theta_{k}+\lambda+1\right)  }%
{\displaystyle\int\limits_{0}^{\infty}}
e^{-nt}t^{^{k+2\mu\theta_{k}+\lambda+1}}dt\nonumber\\
& =\frac{1}{e_{\mu}\left(  nx\right)  A\left(  1\right)  }\sum_{k=0}^{\infty
}\frac{p_{k}\left(  nx\right)  }{\gamma_{\mu}\left(  k\right)  }%
\frac{n^{k+2\mu\theta_{k}+\lambda+1}}{\Gamma\left(  k+2\mu\theta_{k}%
+\lambda+1\right)  }\frac{\Gamma\left(  k+2\mu\theta_{k}+\lambda+2\right)
}{n^{k+2\mu\theta_{k}+\lambda+2}}\nonumber\\
& =\frac{1}{ne_{\mu}\left(  nx\right)  A\left(  1\right)  }\sum_{k=0}^{\infty
}\frac{p_{k}\left(  nx\right)  }{\gamma_{\mu}\left(  k\right)  }\left(
k+2\mu\theta_{k}+\lambda+1\right) \nonumber\\
& =\frac{1}{ne_{\mu}\left(  nx\right)  A\left(  1\right)  }\sum_{k=0}^{\infty
}\frac{p_{k+1}\left(  nx\right)  }{\gamma_{\mu}\left(  k\right)  }%
+\frac{\lambda+1}{ne_{\mu}\left(  nx\right)  A\left(  1\right)  }\sum
_{k=0}^{\infty}\frac{p_{k}\left(  nx\right)  }{\gamma_{\mu}\left(  k\right)
}.\nonumber
\end{align}
By considering the equalities $\left(  \ref{1}\right)  $ and $\left(
\ref{2.2}\right)  ,$ we obtain%
\begin{align*}
& D_{n}^{\ast}\left(  t;x\right) \\
& =\frac{\left(  nx\right)  A\left(  1\right)  e_{\mu}\left(  nx\right)  +\mu
e_{\mu}\left(  -nx\right)  \left[  A\left(  1\right)  -A\left(  -1\right)
\right]  +A^{\prime}\left(  1\right)  e_{\mu}\left(  nx\right)  }{ne_{\mu
}\left(  nx\right)  A\left(  1\right)  }+\frac{\lambda+1}{n}\\
& =x+\frac{\mu}{n}\left\{  \frac{e_{\mu}\left(  -nx\right)  }{e_{\mu}\left(
nx\right)  }\left[  \frac{A\left(  1\right)  -A\left(  -1\right)  }{A\left(
1\right)  }\right]  \right\}  +\frac{1}{n}\left\{  \frac{A^{\prime}\left(
1\right)  }{A\left(  1\right)  }+\lambda+1\right\}  .
\end{align*}
Similarly, for $f\left(  t\right)  =t^{2},$ by means of the equalities
(\ref{1}), (\ref{2.2}) and (\ref{2.3*}), it is seen that the equality
(\ref{2.0-c}) holds.
\end{proof}

\begin{lemma}
\label{lemma 2.3} For each $x\in\left[  0,\infty\right)  ,$ it follows from
the results in Lemma \ref{lemma 2.2}
\begin{align*}
\Omega_{n}^{1}(x)  & :=D_{n}^{\ast}\left(  \left(  t-x\right)  ;x\right) \\
& =\frac{\mu}{n}\left\{  \frac{e_{\mu}\left(  -nx\right)  }{e_{\mu}\left(
nx\right)  }\left[  \frac{A\left(  1\right)  -A\left(  -1\right)  }{A\left(
1\right)  }\right]  \right\}  +\frac{1}{n}\left\{  \frac{A^{\prime}\left(
1\right)  }{A\left(  1\right)  }+\lambda+1\right\}  ,
\end{align*}%
\begin{align}
\Omega_{n}^{2}(x)  & :=D_{n}^{\ast}\left(  \left(  t-x\right)  ^{2};x\right)
\nonumber\\
& =\frac{2x}{n}\left\{  1+\mu\frac{e_{\mu}\left(  -nx\right)  }{e_{\mu}\left(
nx\right)  }\left(  \frac{2A\left(  -1\right)  -A\left(  1\right)  }{A\left(
1\right)  }\right)  \right\} \nonumber\\
& +\frac{1}{n^{2}}\frac{e_{\mu}\left(  -nx\right)  }{e_{\mu}\left(  nx\right)
}\left\{  2\mu\frac{A^{\prime}\left(  1\right)  +A^{\prime}\left(  -1\right)
}{A\left(  1\right)  }+\mu\left(  2\lambda+3\right)  \left(  \frac{A\left(
1\right)  -A\left(  -1\right)  }{A\left(  1\right)  }\right)  \right\}
\nonumber\\
& +\frac{1}{n^{2}}\left\{  \frac{A^{\prime\prime}\left(  1\right)  }{A\left(
1\right)  }+2\left(  \lambda+2\right)  \frac{A^{\prime}\left(  1\right)
}{A\left(  1\right)  }+\left(  \lambda+1\right)  \left(  \lambda+2\right)
\right\} \nonumber\\
& +\frac{2\mu^{2}}{n^{2}}\left(  \frac{A\left(  1\right)  -A\left(  -1\right)
}{A\left(  1\right)  }\right)  .\label{2.7*}%
\end{align}

\end{lemma}

\begin{theorem}
Let $D_{n}^{\ast}$ be the operators given by $\left(  \ref{appell-dunkl-s-d}%
\right)  .$ Then, for any $f\in C\left[  0,\infty\right)  \cap E,$ the
following relation holds%
\[
\lim_{n\rightarrow\infty}D_{n}^{\ast}\left(  f;x\right)  =f\left(  x\right)  ,
\]
uniformly on each compact subset of $\left[  0,\infty\right)  ,$ where%
\[
E:=\left\{  f:~x\in\left[  0,\infty\right)  ,\frac{~f\left(  x\right)
}{1+x^{2}}\text{ is convergent as }x\rightarrow\infty\right\}  .
\]

\end{theorem}

\begin{proof}
From the results in Lemma \ref{lemma 2.2}%
\[
\lim_{n\rightarrow\infty}D_{n}^{\ast}\left(  t^{i};x\right)  =x^{i},~i=0,1,2,
\]
holds where the convergence holds uniformly in each compact subset of $\left[
0,\infty\right)  .$ Then, applying the universal Korovkin type Theorem 4.1.4
(vi) given in \cite{altomare} gives the desired result.
\end{proof}

\section{Rates of Convergence}

In this part, we calculate the order of approximation by means of the usual
and second modulus of continuity and Lipschitz class functions. First of all,
we recall some definitions as follows.

Let $f\in\widetilde{C}[0,\infty)$ and $\delta>0.$ The modulus of continuity of
$f$ denoted by $\omega\left(  f;\delta\right)  $ is defined by
\[
\omega\left(  f;\delta\right)  :=\sup_{\substack{x,y\in\lbrack0,\infty)
\\\left\vert x-y\right\vert \leq\delta}}\left\vert f\left(  x\right)
-f\left(  y\right)  \right\vert
\]
where $\widetilde{C}[0,\infty)~$is the space of uniformly continuous functions
on $[0,\infty).~$Then$,$ for any $\delta>0$ and each $x\in\lbrack0,\infty)$,
we have the following inequality%
\begin{equation}
\left\vert f\left(  x\right)  -f\left(  y\right)  \right\vert \leq
\omega\left(  f;\delta\right)  \left(  \frac{\left\vert x-y\right\vert
}{\delta}+1\right)  .\label{a8}%
\end{equation}

Let $C_{B}\left[  0,\infty\right)  ~$be the class of real valued functions
defined on $\left[  0,\infty\right)  $ which are bounded and uniformly
continuous with the norm $\left\Vert f\right\Vert _{C_{B}}=\sup_{x\in\left[
0,\infty\right)  }\left\vert f\left(  x\right)  \right\vert .$ The second
modulus of continuity of $f\in C_{B}\left[  0,\infty\right)  $ is defined by%
\[
\omega_{2}\left(  f;\delta\right)  :=\sup_{0<t\leq\delta}\left\Vert f\left(
.+2t\right)  -2f\left(  .+t\right)  +f\left(  .\right)  \right\Vert _{C_{B}}.
\]
Now, let us give the following definitions.

\begin{definition}
Let $f$ be a real valued continuous function defined on $[0,\infty).$ Then $f$
is said to be Lipschitz continuous of order $\gamma$ on $[0,\infty)$ if%
\[
\left\vert f\left(  x\right)  -f\left(  y\right)  \right\vert \leq M\left\vert
x-y\right\vert ^{\gamma}%
\]
for $x,y\in\lbrack0,\infty)$ with $M>0$ and $0<\gamma\leq1.$ The set of
Lipschitz continuous functions is denoted by Lip$_{M}\left(  \gamma\right)  .$
\end{definition}

\begin{definition}
\cite{ditzian-totik} Peetre's $K$-functional of the function $f\in
C_{B}\left[  0,\infty\right)  $ is defined by%
\begin{equation}
K\left(  f;\delta\right)  :=\inf_{g\in C_{B}^{2}\left[  0,\infty\right)
}\left\{  \left\Vert f-g\right\Vert _{_{C_{B}}}+\delta\left\Vert g\right\Vert
_{C_{B}^{2}}\right\} \label{a11}%
\end{equation}
where%
\[
C_{B}^{2}\left[  0,\infty\right)  :=\left\{  g\in C_{B}\left[  0,\infty
\right)  :g^{\prime},\text{ }g^{\prime\prime}\in C_{B}\left[  0,\infty\right)
\right\}
\]
and the norm%
\[
\left\Vert g\right\Vert _{C_{B}^{2}}:=\left\Vert g\right\Vert _{C_{B}%
}+\left\Vert g^{\prime}\right\Vert _{C_{B}}+\left\Vert g^{\prime\prime
}\right\Vert _{C_{B}}.
\]
It is clear that the following inequality%
\begin{equation}
K\left(  f;\delta\right)  \leq M\left\{  \omega_{2}\left(  f;\sqrt{\delta
}\right)  +\min\left(  1,\delta\right)  \left\Vert f\right\Vert _{C_{B}%
}\right\} \label{a14}%
\end{equation}
holds for all $\delta>0.$ The constant $M$ is independent of $f$ and $\delta.$
\end{definition}

\begin{theorem}
\label{theorem 3.1}For $f\in\widetilde{C}[0,\infty)\cap E,$ we have
\end{theorem}%

\[
\left\vert D_{n}^{\ast}\left(  f;x\right)  -f\left(  x\right)  \right\vert
\leq2\omega\left(  f;\sqrt{\Omega_{n}^{2}\left(  x\right)  }\right)  ,
\]
where $\Omega_{n}^{2}$ is given as in Lemma \ref{lemma 2.3}.

\begin{proof}
From linearity and positivity of the operators $D_{n}^{\ast},$ by applying
(\ref{a8}), we get%
\begin{align}
& \left\vert D_{n}^{\ast}\left(  f;x\right)  -f\left(  x\right)  \right\vert
\nonumber\\
& \leq\frac{1}{e_{\mu}\left(  nx\right)  A\left(  1\right)  }\sum
_{k=0}^{\infty}\frac{p_{k}\left(  nx\right)  }{\gamma_{\mu}\left(  k\right)
}\frac{n^{k+2\mu\theta_{k}+\lambda+1}}{\Gamma\left(  k+2\mu\theta_{k}%
+\lambda+1\right)  }\nonumber\\
& \times%
{\displaystyle\int\limits_{0}^{\infty}}
e^{-nt}t^{^{k+2\mu\theta_{k}+\lambda}}\left\vert f\left(  t\right)  -f\left(
x\right)  \right\vert dt\nonumber\\
& \leq\frac{1}{e_{\mu}\left(  nx\right)  A\left(  1\right)  }\sum
_{k=0}^{\infty}\frac{p_{k}\left(  nx\right)  }{\gamma_{\mu}\left(  k\right)
}\frac{n^{k+2\mu\theta_{k}+\lambda+1}}{\Gamma\left(  k+2\mu\theta_{k}%
+\lambda+1\right)  }\label{2.7.a}\\
& \times%
{\displaystyle\int\limits_{0}^{\infty}}
e^{-nt}t^{^{k+2\mu\theta_{k}+\lambda}}\left(  \frac{\left\vert t-x\right\vert
}{\delta}+1\right)  \omega\left(  f;\delta\right)  dt\nonumber\\
& \leq\left\{  1+\frac{1}{\delta}\frac{1}{e_{\mu}\left(  nx\right)  A\left(
1\right)  }\sum_{k=0}^{\infty}\frac{p_{k}\left(  nx\right)  }{\gamma_{\mu
}\left(  k\right)  }\frac{n^{k+2\mu\theta_{k}+\lambda+1}}{\Gamma\left(
k+2\mu\theta_{k}+\lambda+1\right)  }\right. \nonumber\\
& \left.  \times%
{\displaystyle\int\limits_{0}^{\infty}}
e^{-nt}t^{^{k+2\mu\theta_{k}+\lambda}}\left\vert t-x\right\vert dt\right\}
\omega\left(  f;\delta\right)  .\nonumber
\end{align}
From the Cauchy-Schwarz inequality for integration, one may write%
\[%
{\displaystyle\int\limits_{0}^{\infty}}
e^{-nt}t^{^{k+2\mu\theta_{k}+\lambda}}\left\vert t-x\right\vert dt\leq\left(
\frac{\Gamma\left(  k+2\mu\theta_{k}+\lambda+1\right)  }{n^{k+2\mu\theta
_{k}+\lambda+1}}\right)  ^{1/2}\left(
{\displaystyle\int\limits_{0}^{\infty}}
e^{-nt}t^{^{k+2\mu\theta_{k}+\lambda}}\left(  t-x\right)  ^{2}dt\right)
^{1/2},
\]
by using this inequality, it follows that%
\begin{align}
& \sum_{k=0}^{\infty}\frac{p_{k}\left(  nx\right)  }{\gamma_{\mu}\left(
k\right)  }\frac{n^{k+2\mu\theta_{k}+\lambda+1}}{\Gamma\left(  k+2\mu
\theta_{k}+\lambda+1\right)  }%
{\displaystyle\int\limits_{0}^{\infty}}
e^{-nt}t^{^{k+2\mu\theta_{k}+\lambda}}\left\vert t-x\right\vert dt\nonumber\\
& \leq\sum_{k=0}^{\infty}\frac{p_{k}\left(  nx\right)  }{\gamma_{\mu}\left(
k\right)  }\left(  \frac{n^{k+2\mu\theta_{k}+\lambda+1}}{\Gamma\left(
k+2\mu\theta_{k}+\lambda+1\right)  }\right)  ^{^{1/2}}\left(
{\displaystyle\int\limits_{0}^{\infty}}
e^{-nt}t^{^{k+2\mu\theta_{k}+\lambda}}\left(  t-x\right)  ^{2}dt\right)
^{1/2}.\label{2.8}%
\end{align}
If we now apply Cauchy-Schwarz inequality for sum on the right hand side of
(\ref{2.8}), we get%
\begin{align}
& \sum_{k=0}^{\infty}\frac{p_{k}\left(  nx\right)  }{\gamma_{\mu}\left(
k\right)  }\frac{n^{k+2\mu\theta_{k}+\lambda+1}}{\Gamma\left(  k+2\mu
\theta_{k}+\lambda+1\right)  }%
{\displaystyle\int\limits_{0}^{\infty}}
e^{-nt}t^{^{k+2\mu\theta_{k}+\lambda}}\left\vert t-x\right\vert dt\nonumber\\
& \leq\sqrt{e_{\mu}\left(  nx\right)  A\left(  1\right)  }\left(  e_{\mu
}\left(  nx\right)  A\left(  1\right)  D_{n}^{\ast}\left(  \left(  t-x\right)
^{2};x\right)  \right)  ^{1/2}\nonumber\\
& =e_{\mu}\left(  nx\right)  A\left(  1\right)  \left(  D_{n}^{\ast}\left(
\left(  t-x\right)  ^{2};x\right)  \right)  ^{1/2}\nonumber\\
& =e_{\mu}\left(  nx\right)  A\left(  1\right)  \left(  \Omega_{n}^{2}\left(
x\right)  \right)  ^{1/2},\label{a9}%
\end{align}
where $\Omega_{n}^{2}\left(  x\right)  $ is as in the equality (\ref{2.7*}).
When we consider (\ref{a9}) in (\ref{2.7.a}), we obtain%
\[
\left\vert D_{n}^{\ast}\left(  f;x\right)  -f\left(  x\right)  \right\vert
\leq\left\{  1+\frac{1}{\delta}\sqrt{\Omega_{n}^{2}\left(  x\right)
}\right\}  \omega\left(  f;\delta\right)  .
\]
If we choose $\delta=\sqrt{\Omega_{n}^{2}\left(  x\right)  }$, we arrive at
\[
\left\vert D_{n}^{\ast}\left(  f;x\right)  -f\left(  x\right)  \right\vert
\leq2\omega\left(  f;\sqrt{\Omega_{n}^{2}\left(  x\right)  }\right)  .
\]

We note that $\Omega_{n}^{2}\left(  x\right)  $ goes to zero when
$n\rightarrow\infty$.
\end{proof}

\begin{theorem}
For $f\in Lip_{M}\left(  \alpha\right)  ,$ such that $0<\alpha\leq1,$ $M\in%
\mathbb{R}
^{+}$ we have
\end{theorem}%

\[
\left\vert D_{n}^{\ast}\left(  f;x\right)  -f\left(  x\right)  \right\vert
\leq M\left(  \Omega_{n}^{2}\left(  x\right)  \right)  ^{\frac{\alpha}{2}},
\]
where $\Omega_{n}^{2}$ is given in Lemma \ref{lemma 2.3}.

\begin{proof}
Since $f\in Lip_{M}\left(  \alpha\right)  ,$ we can write from linearity%
\[
\left\vert D_{n}^{\ast}\left(  f;x\right)  -f\left(  x\right)  \right\vert
\leq D_{n}^{\ast}\left(  \left\vert f\left(  t\right)  -f\left(  x\right)
\right\vert ;x\right)  \leq MD_{n}^{\ast}\left(  \left\vert t-x\right\vert
^{\alpha};x\right)  .
\]
By taking into account Lemma \ref{lemma 2.3} and H\"{o}lder inequality, we get%
\[
\left\vert D_{n}^{\ast}\left(  f;x\right)  -f\left(  x\right)  \right\vert
\leq M\left(  \Omega_{n}^{2}\left(  x\right)  \right)  ^{\frac{\alpha}{2}},
\]
which ends the proof.
\end{proof}

Now, we give rate of convergence of the operators $D_{n}^{\ast}$ via Peetre's K-functional.

\begin{lemma}
\label{lemma 4} For any $g\in C_{B}^{2}\left[  0,\infty\right)  $, we have%
\[
\left\vert D_{n}^{\ast}\left(  g;x\right)  -g\left(  x\right)  \right\vert
\leq\lambda_{n}\left(  x\right)  \left\Vert g\right\Vert _{C_{B}^{2}\left[
0,\infty\right)  }%
\]
where%
\begin{equation}
\lambda_{n}\left(  x\right)  =\Omega_{n}^{1}(x)+\frac{\Omega_{n}^{2}(x)}%
{2}.\label{2.9}%
\end{equation}

\end{lemma}

\begin{proof}
From the Taylor's series of the function $g\in C_{B}^{2}\left[  0,\infty
\right)  ,$ we can write%
\[
g\left(  t\right)  =g\left(  x\right)  +g^{\prime}\left(  x\right)  \left(
t-x\right)  +\left(  t-x\right)  ^{2}\frac{g^{\prime\prime}\left(  \xi\right)
}{2!},~\xi\in\left(  x,t\right)  .
\]
By operating by $D_{n}^{\ast}$ on both sides of this equality and then using
the linearity of the operator, we get%
\[
D_{n}^{\ast}\left(  g;x\right)  -g\left(  x\right)  =g^{\prime}\left(
x\right)  D_{n}^{\ast}\left(  \left(  t-x\right)  ;x\right)  +\frac
{g^{\prime\prime}\left(  \xi\right)  }{2}D_{n}^{\ast}\left(  \left(
t-x\right)  ^{2};x\right)  .
\]
By considering Lemma \ref{lemma 2.3}, one can have%
\begin{align*}
\left\vert D_{n}^{\ast}\left(  g;x\right)  -g\left(  x\right)  \right\vert  &
\leq\left\vert g^{\prime}\left(  x\right)  D_{n}^{\ast}\left(  \left(
t-x\right)  ;x\right)  \right\vert +\left\vert \frac{g^{\prime\prime}\left(
\xi\right)  }{2}D_{n}^{\ast}\left(  \left(  t-x\right)  ^{2};x\right)
\right\vert \\
& \leq\left\Vert g^{\prime}\right\Vert _{C_{B}\left[  0,\infty\right)  }%
\Omega_{n}^{1}(x)+\frac{1}{2}\left\Vert g^{\prime\prime}\right\Vert
_{C_{B}\left[  0,\infty\right)  }\Omega_{n}^{2}(x)\\
& \leq\left(  \Omega_{n}^{1}(x)+\frac{\Omega_{n}^{2}(x)}{2}\right)  \left\Vert
g\right\Vert _{C_{B}^{2}\left[  0,\infty\right)  }.
\end{align*}
So the proof is completed.
\end{proof}

\begin{theorem}
For any $f\in C_{B}\left[  0,\infty\right)  $, we have%
\[
\left\vert D_{n}^{\ast}\left(  f;x\right)  -f\left(  x\right)  \right\vert
\leq2M\left\{  \min\left(  1,\frac{\lambda_{n}\left(  x\right)  }{2}\right)
\left\Vert f\right\Vert _{C_{B}\left[  0,\infty\right)  }+\omega_{2}\left(
f;\sqrt{\frac{\lambda_{n}\left(  x\right)  }{2}}\right)  \right\}  ,
\]
where $M$ is a positive constant which is independent of $n$ and $\lambda
_{n}\left(  x\right)  $ given by $\left(  \ref{2.9}\right)  .$
\end{theorem}

\begin{proof}
Let $g\in C_{B}^{2}\left[  0,\infty\right)  .$ In view of Lemma \ref{lemma 4},
one can have%
\begin{align*}
\left\vert D_{n}^{\ast}\left(  f;x\right)  -f\left(  x\right)  \right\vert  &
\leq\left\vert D_{n}^{\ast}\left(  f-g;x\right)  \right\vert +\left\vert
D_{n}^{\ast}\left(  g;x\right)  -g\left(  x\right)  \right\vert +\left\vert
f\left(  x\right)  -g\left(  x\right)  \right\vert \\
& \leq2\left\Vert f-g\right\Vert _{C_{B}\left[  0,\infty\right)  }+\lambda
_{n}\left(  x\right)  \left\Vert g\right\Vert _{C_{B}^{2}\left[
0,\infty\right)  }\\
& =2\left\{  \left\Vert f-g\right\Vert _{C_{B}\left[  0,\infty\right)  }%
+\frac{\lambda_{n}\left(  x\right)  }{2}\left\Vert g\right\Vert _{C_{B}%
^{2}\left[  0,\infty\right)  }\right\}  ,
\end{align*}
from $\left(  \ref{a11}\right)  $ it follows%
\[
\left\vert D_{n}^{\ast}\left(  f;x\right)  -f\left(  x\right)  \right\vert
\leq2K\left(  f;\frac{\lambda_{n}\left(  x\right)  }{2}\right)  .
\]
By $\left(  \ref{a14}\right)  $, we obtain%
\[
\left\vert D_{n}^{\ast}\left(  f;x\right)  -f\left(  x\right)  \right\vert
\leq2M\left\{  \min\left(  1,\frac{\lambda_{n}\left(  x\right)  }{2}\right)
\left\Vert f\right\Vert _{C_{B}\left[  0,\infty\right)  }+\omega_{2}\left(
f;\sqrt{\frac{\lambda_{n}\left(  x\right)  }{2}}\right)  \right\}  ,
\]
which completes the proof.
\end{proof}


\begin{thebibliography}{99}                                                                                               %
\bibitem {aktas-cekim-tasdelen}Akta\c{s}, R., \c{C}ekim, B. and Ta\c{s}delen,
F., A Dunkl analogue of operators including two-variable Hermite polynomials,
Bull. Malays. Math. Sci. Soc. (2018). https://doi.org/10.1007/s40840-018-0631-z

\bibitem {altomare}Altomare, F. \ and Campiti, M., Korovkin-type approximation
theory and its applications. Appendix A by Michael Pannenberg and Appendix B
by Ferdinand Beckhoff, de Gruyter Studies in Mathematics, 17. Walter de
Gruyter \& Co., Berlin, 1994.

\bibitem {ben cheikh 2007}Ben Cheikh, Y. and Gaied, M., Dunkl-Appell
$d$-ortogonal polynomials, Integral Transforms and Special Functions, 18 (8)
(2007), 581-597.

\bibitem {ben cheikh2014}Ben Cheikh, Y., Gaied, M. and Zaghouani, M., A
$q$-Dunkl-classical $q$-Hermite type polynomials, Georgian Math. J., 21(2)
(2014), 125--137.

\bibitem {C}Chihara, T.S., An Introduction to Orthogonal Polynomials, Gordon
and Breach, New York, 1978.

\bibitem {CI}Ciupa, A., A class of integral Favard-Szasz type operators,
Studia Univ. Babe\c{s}-Bolyai Math., 40(1) (1995), 39-47.

\bibitem {deshwal-agrawal-araci}Deshwal, S., Agrawal, P. N. and Arac\i, S.,
Dunkl analogue of Sz\'{a}sz Mirakyan Operators of Blending Type, Open
Mathematics, 16 (1) (2017), doi: 10.1515/math-2018-0116.

\bibitem {ditzian-totik}Ditzian Z. and Totik, V., Moduli of smoothness, volume
9 of Springer Series in Computational Mathematics, Springer-Verlag, New York, 1987.

\bibitem {dunkl1991}Dunkl, C. F., Integral Kernels with Reflection Group
Invariance, Canad. J. Math., 43 (6), (1991), 1213-1227.

\bibitem {dunkl1994}Dunkl, C. F., De Jeu, M. F. E. and Opdam, E. M., Singular
Polynomials for Finite Reflection Groups, Transactions of the American
Mathematical Society, 346 (1) (1994), 237-256.

\bibitem {icoz-cekim2015}I\c{c}\"{o}z, G. and \c{C}ekim, B., Dunkl
generalization of Sz\'{a}sz operators via $q$-calculus, J. Inequal. Appl.,
2015: 284 (2015).

\bibitem {icoz}I\c{c}\"{o}z, G. and \c{C}ekim, B., Stancu-type generalization
of Dunkl analogue of Sz\'{a}sz-Kantorovich operators. Math. Methods Appl.
Sci., 39 (2016), 1803--1810.

\bibitem {ismail}Ismail, M. E. H., On a generalization of Szasz operators,
Mathematica (Cluj), 39 (1974), 259-267.

\bibitem {J-L}Jakimovski, A. and Leviatan, D. Generalized Szasz operators for
the approximation in the infinite interval, Mathematica (Cluj), 11 (1969), 97-103.

\bibitem {kajla-agrawal}Kajla, A. and Agrawal, P.N., Sz\'{a}sz-Durrmeyer type
operators based on Charlier polynomials, Appl. Math. Comput., 268 (2015), 1001--1014.

\bibitem {MMS}Mishra, V.N., Mursaleen, M. and Sharma, P., Some approximation
properties of Baskakov-Szasz Stancu operators, Appl. Math. Inf. Sci., 9 (6)
(2015), 3159-3167.

\bibitem {mursaleen-ensari2015}Mursaleen, M. and Khursheed, J. Ansari, On
Chlodowsky variant of Sz\'{a}sz operators by Brenke type polynomials, Appl.
Math. Comput., 271(2015), 991--1003.

\bibitem {mursaleen-}Mursaleen, M., Khan, T. and Nasiruzzaman, Md.,
Approximating properties of generalized Dunkl analogue of Szasz operators,
Appl. Math. Inf. Sci., 10 (6) (2016), 2303-2310.

\bibitem {olgun-ince-tasdelen}Olgun, A., \.{I}nce H. G. and Ta\c{s}delen, F.,
Kantorovich-type generalization of Meyer -K\"{o}nig and Zeller operators via
generating functions, An. \c{S}t. Univ. Ovidius Constanta, 21(3) (2013), 209-221.

\bibitem {Rao-wafi-ana acu d}Rao, N., Wafi, A. and Acu, A. M.,
q-Sz\'{a}sz--Durrmeyer type operators based on Dunkl analogue, Complex
Analysis and Operator Theory, doi.org/10.1007/s11785-018-0816-3.

\bibitem {rosenblum}Rosenblum, M., Generalized Hermite polynomials and the
Bose-like oscillator calculus, Oper. Theory Adv. Appl., 73 (1994), 369--396.

\bibitem {sucu}Sucu, S., Dunkl analogue of Sz\'{a}sz operators, Appl. Math.
Comput., 244 (2014), 42--48.

\bibitem {tasdelen-aktas-altin}Ta\c{s}delen, F., Akta\c{s}, R. and Alt\i n,
A., A Kantrovich type of Sz\'{a}sz operators including Brenke type
polynomials, Abstr. Appl. Anal., 2012 (2012), 13 pages.

\bibitem {VIS}Varma, S., Sucu, S. and \.{I}\c{c}\"{o}z, G., Generalization of
Szasz operators involving Brenke type polynomials, Comput. Math. Appl., 64 (2)
(2012), 121-127.

\bibitem {VTT}Varma S. and Ta\c{s}delen, F., On a generalization of
Szasz-Durrmeyer operators with some orthogonal polynomials, Stud. Univ.
Babe\c{s}-Bolyai Math., 58 (2) (2013), 225-232.

\bibitem {wafi-rao}Wafi, A. and Rao, N., Szasz-Gamma operators based on Dunkl
analogue, Iranian Journal of Science and Technology, Transactions A: Science
(2017), https://doi.org/10.1007/s40995-017-0433-4
\end{thebibliography}
\end{document}